\documentclass{scrartcl}
\usepackage[margin=2cm,bottom=1cm,includefoot]{geometry}
\usepackage{amsmath,amssymb,amsthm}
\usepackage{bbm}
\usepackage{mathrsfs}
\usepackage{enumerate}
\usepackage{graphicx}
\usepackage{hyperref}
\usepackage[all]{hypcap}
\usepackage[utf8]{inputenc}
\usepackage{float}
\usepackage{multicol}

\newtheorem{thm}{Theorem}
\newtheorem{lem}[thm]{Lemma}
\newtheorem{prop}[thm]{Proposition}
\newtheorem{cor}[thm]{Corollary}

\let\BFseries\bfseries\def\bfseries{\BFseries\mathversion{bold}}

\newcommand{\N}{\mathbb{N}}
\newcommand{\R}{\mathbb{R}}
\newcommand{\E}{\mathbb{E}}
\newcommand{\p}{\mathbb{P}}
\newcommand{\FF}{\mathcal{F}}
\newcommand{\CC}{\mathcal{C}}
\newcommand{\BB}{\mathcal{B}}
\newcommand{\1}{\mathbbm{1}}
\newcommand{\eps}{\varepsilon}
\DeclareMathOperator{\e}{e}
\newcommand{\D}{\mathrm{d}}
\newcommand{\op}{\operatorname}
\newcommand{\y}{\|}
\newcommand{\lne}{<}
\newcommand{\gne}{>}
\newcommand{\abs}{|}

\begin{document}
\title{Brownian Motion Conditioned to Spend Limited Time Below a Barrier}
\author{Frank Aurzada\footnote{Technical University of Darmstadt, Department of Mathematics
\newline E-mail: aurzada@mathematik.tu-darmstadt.de, schickentanz@mathematik.tu-darmstadt.de
}\and Dominic T. Schickentanz\footnotemark[1]}
\date{\today}
\maketitle
\begin{abstract}
We condition a Brownian motion with arbitrary starting point $y \in \R$ on spending at most~$1$~time unit below~$0$ and provide an explicit description of the resulting process. In particular, we provide explicit formulas for the distributions of its last zero~$g=g^y$ and of its occupation time~$\Gamma=\Gamma^y$ below~$0$ as functions of $y$. This generalizes Theorem~4 of \cite{BB11}, which covers the special case $y=0$. Additionally, we study the behavior of the distributions of~$g^y$~and~$\Gamma^y$, respectively, for $y \to \pm\infty$.
\end{abstract}
\section{Introduction}
Conditioning stochastic processes on avoiding certain sets is a classical problem in probability theory. In 1957, Doob proved that a Brownian motion starting in $y \gne 0$ which is conditioned to avoid the negative half-line is nothing but a three-dimensional Bessel process starting in~$y$ (see \cite{Doo57}).
A more modern presentation of this result can be found in \cite{Pit75}.
Proceeding from Doob's work, similar problems have been considered for more complicated processes and more complicated or time-dependent sets to be avoided. Many examples are referred to in the introduction of \cite{Bar20}.
\\
In the present paper, we advance in a different direction: We allow a Brownian motion with an arbitrary starting point to spend limited time in the negative half-line. More precisely, fix $y \in \R$ and let $B=(B_t)_{t \ge 0}$ be a Brownian motion starting in~$y$. For each $T \ge 0$, let $$\Gamma_T:= \int_0^T \1_{\{B_s \lne 0\}} \D s$$ be the time $B$ spends below~$0$ until time $T$. In Theorem~\ref{thm1}, we will show that $$\p(B \in \cdot\,\, \abs\, \Gamma_T \le 1)$$ convergences weakly in $\CC([0,\infty))$ for $T \to \infty$.
The limiting process can informally be described as follows: Up to a random time $g$, the distribution of which we will determine explicitly, the limiting process is a conditioned Brownian bridge starting in~$y$ and ending in~$0$. Afterwards, it is a three-dimensional Bessel process. For $y \gne 0$, the event $\{g=0\}$ occurs with positive probability. If it occurs, the Bessel process starts immediately so that the limiting process stays positive all the time. A rigorous construction is given in Section~\ref{Sec2}.
\\
Even though we allow the limiting process to spend a total of~$1$~time unit below~$0$, its actual total occupation time $\Gamma$ below~$0$ is strictly smaller than~$1$ almost surely. In Theorem~\ref{thm2}, we will explicitly determine the distribution of $\Gamma$.
\\
Our Theorems~\ref{thm1}~and~\ref{thm2} generalize Theorem~4 of \cite{BB11}, which covers the special case $y=0$. The main upshot of the present paper is the complete understanding of the limiting process for general $y \in \R$. In particular, we describe the distributions of the last zero $g$ and the occupation time $\Gamma$ below~$0$ of the limiting process as explicit functions of~$y$.
\\
Clearly, the problem we consider is equivalent to conditioning a standard Brownian motion on spending at most $1$~time unit below the barrier $-y$, which explains the title of this paper. In view of the scaling property of Brownian motion, it is straightforward to replace the single time unit allowed to spend below the barrier by any other amount $c \gne 0$ of time.
\\
The phenomenon that the condition is not exhausted completely can be seen as an instance of entropic repulsion and occurs frequently when conditioning stochastic processes on negligible events. Examples related to the present work include Brownian motion with restricted local time in~$0$ (see \cite{RVY06,BB11,KS16}), with uniformly bounded local time in every point (see \cite{BB10}), with bounded excursion lengths (see \cite{RVY09}) and, of course, conditioned to stay positive.
\\
Without relying on the explicit distributions, Proposition \ref{gGamma} will provide an identity connecting the distributions of $g$, $\Gamma$ and the first entrance time of the limiting process to the negative half-line.
\\
Finally, we will take a closer look at the behavior of the distributions of~$g^y$~and~$\Gamma^y$ as functions of~$y$. For $y \to -\infty$, the transformed occupation time $y^2(1-\Gamma^y)$ is approximately exponentially distributed while the weak limit of $y^2(1-g^y)$ has some other explicit distribution (see part~(b) of Theorem~\ref{thm3}). In particular, $\Gamma^y$ and~$g^y$ both converge weakly to~$1$ for $y \to -\infty$. For $y \gne 0$,
we have to take into account that the limiting process may stay positive permanently.
Conditioned on spending time below~$0$ at all, the distribution of $\Gamma^y$ is independent of~$y$ for $y \ge 0$ while the conditional distribution of~$\frac{g^y}{y^2}$ converges weakly to an inverse chi-squared distribution for $y \to \infty$ (see part~(a) of Theorem~\ref{thm3}). In particular, conditional on the existence of a zero, the last zero diverges weakly to~$\infty$ for $y \to \infty$.
\\
The outline of this paper is as follows. In Section~\ref{Sec2}, we will formulate the main results rigorously.  Section~\ref{Sec3} is concerned with the distribution of $g$. In particular, we check that $g$ is well-defined which, unlike in the special case $y=0$ covered in \cite{BB11}, is non-trivial. The subsequent two sections are devoted to the proof of Theorem~\ref{thm1}. The key part will be Proposition~\ref{gConv}: It states that the conditional distribution of $$g_T:= \max\{s \in [0,T]: B_s=0\},$$ the last zero of $B$ before time $T$, converges to the distribution of $g$, the last zero of the limiting process, in total variation for $T \to \infty$. In Section~\ref{Sec6}, we will finally prove the remaining results.

\section{Main Results}
\label{Sec2}
In the introduction, we gave an overview of the main results and an informal description of the limiting process. Before we make this rigorous, we introduce an auxiliary notation: Let
\begin{equation}
\label{qDef}
q(t,u):=q^y(t,u):= \p\left(\int_0^t \1_{\{b_s' \lne 0\}} \D s\le u\right), \quad t,u \ge 0,
\end{equation}
be the probability that a Brownian bridge $(b_s')_{s \in [0,t]}$ of length~$t$ with $b_0'=y$ and $b_t'=0$ spends at most~$u$~time units below~$0$. Explicit formulas for $q$ are given in \eqref{q=1}, \eqref{qFormEqn} and \eqref{qFormPec}.
\\
Now we are ready to introduce the limiting process~$X$ rigorously:
\begin{enumerate}
\item Depending on the sign of~$y$, let~$g=g^y$ be a non-negative random variable with
\begin{align}
\label{gForm}
\p(g \le x) =
\begin{cases}
\1_{\{x \le 1\}} \dfrac{\sqrt{x}}{2} + \1_{\{x \gne 1\}} \left(1-\dfrac{1}{2 \sqrt{x}}\right),& \quad x \ge 0,\ y=0,\bigskip\\
\dfrac{\int_0^x q^y(t,1) \frac{1}{\sqrt{t}}\e^{-\frac{y^2}{2t}} \D t}{2\int_0^1 \frac{1}{\sqrt{t}}\e^{-\frac{y^2}{2t}} \D t},& \quad x \ge 0,\ y\lne 0,\bigskip\\
\dfrac{2\sqrt{2\pi}y+\int_0^x q^y(t,1) \frac{1}{\sqrt{t}}\e^{-\frac{y^2}{2t}} \D t}{2\sqrt{2\pi}y+4},& \quad x \ge 0,\ y\gne 0.
\end{cases}
\end{align}
We will see in Corollary~\ref{gFinite} that $\p(g \lne \infty)=1$ holds in all three cases. A plot of the distribution function of~$g$ for different values of~$y$ can be found in Figure \ref{Fig1}.
\item Let $b=(b_t)_{t \in [0,1]}$ be a process which, if restricted to $\{g=x\}$ for $x \ge 0$, is a standard Brownian bridge conditioned on $$\int_0^{x} \1_{\left\{\sqrt{x} b_{\frac{s}{x}}+y-\frac{s}{x}y \lne 0\right\}} \D s \le 1.$$
\item Let $Z=(Z_t)_{t \ge 0}$ be a three-dimensional Bessel process starting in~$0$, independent of $(g,b)$.
\item For $y \gne 0$, let $Y=(Y_t)_{t \ge 0}$ be a three-dimensional Bessel process starting in~$y$, independent of $(g,b,Z)$.
\item We define $$X:=(X_t)_{t \ge 0}:= \left(\1_{\{g \gne 0,\, t \lne g\}} \left(\sqrt{g}b_{\frac{t}{g}}+y-\frac{t}{g}y\right) + \1_{\{g \gne 0,\, t \ge g\}} Z_{t-g}+ \1_{\{g=0, y \gne 0\}}Y_t\right)_{t \ge 0}.$$
\end{enumerate}
Endowing the space $\CC([0,\infty))$ with the topology of locally uniform convergence and the corresponding Borel $\sigma$-algebra, we obtain the announced convergence result:
\begin{thm}
\label{thm1}
For $T \to \infty$, the probability measures $\p(B \in \cdot \,\, \abs\, \Gamma_T \le 1)$ converge weakly in $\CC([0,\infty))$ to the distribution of~$X$.
\end{thm}
Now let $$\Gamma:=\Gamma^y:= \int_0^\infty \1_{\{X_s \lne 0\}} \D s = \int_0^g \1_{\{X_s \lne 0\}} \D s$$ be the total time $X$ spends below~$0$. By construction of~$X$, we have $\Gamma \le 1$. The distribution function of $\Gamma$ has the following explicit form:

\begin{thm}
\label{thm2}
We have
\begin{align}
\label{GammaForm}
\p(\Gamma \le u) =
\begin{cases}
\dfrac{\int_0^u \frac{1}{\sqrt{t}} \e^{-\frac{y^2}{2t}} \D t}{\int_0^1 \frac{1}{\sqrt{t}} \e^{-\frac{y^2}{2t}} \D t}, \quad& u \in [0,1],\ y \le 0,\bigskip\\
\dfrac{\sqrt{2 \pi} y + 2 \sqrt{u}}{\sqrt{2 \pi} y + 2}, \quad& u \in [0,1],\ y \ge 0.
\end{cases}
\end{align}
\end{thm}

A plot of this distribution function for different values of~$y$ can be found in Figure \ref{Fig1}. Noting
\begin{equation}
\label{q=1}
q(t,u)=1, \quad u \ge t \ge 0,
\end{equation}
a comparison of~\eqref{gForm}~and~\eqref{GammaForm} yields
\begin{equation}
\label{gGammaForm}
2\p(g \le u) = \p(\Gamma \le u), \quad u \in [0,1],\ y\le 0.
\end{equation}
This identity is not a mere coincidence. Let $$\tau:= \inf\{t \ge 0: X_t \le 0\}$$ be the first entrance time of the limiting process to the negative half-line. Noting $g \gne 0$ and $\tau = 0$ for $y \le 0$, Proposition \ref{gGamma} below provides a generalization of~\eqref{gGammaForm} which is valid for all $y \in \R$. We will prove this result, which is a consequence of the arcsine laws and the strong Markov property, without relying on the explicit formulas~\eqref{gForm}~and~\eqref{GammaForm}.
\begin{prop}
\label{gGamma}
We have $$2 \p(g \in (0,u]) = \p( \tau + \Gamma \leq u), \quad u \in [0,1].$$
\end{prop}
Finally, we discuss the behavior of the distributions of~$g^y$~and~$\Gamma^y$, respectively, for $y \to \pm \infty$. According to~\eqref{gForm}~and~\eqref{GammaForm}, we have $$\p(g^y =0) = \p(\Gamma^y = 0) = \frac{\sqrt{2\pi} y}{\sqrt{2\pi} y +2}\gne 0, \quad y \gne 0.$$ In particular, $g^y$ and $\Gamma^y$ both converge weakly to~$0$ for $y \to \infty$.
Note that the events $\{g^y=0\}$ and $\{\Gamma^y=0\}$ correspond to the situation where the limiting process stays positive all the time.
Formula~\eqref{GammaForm} implies $$\p(\Gamma^y \le u \,\abs\, \Gamma^y \gne 0) = \sqrt{u} = \p(\Gamma^0 \le u), \quad u \in [0,1],\ y \ge 0.$$ Consequently, conditioned on spending time below~$0$ at all, the distribution of the occupation time~$\Gamma^y$ below~$0$ is given by the square of a uniform distribution on $[0,1]$ for each $y \ge 0$. In particular, this conditional distribution is independent of the starting point $y \ge 0$. The reason is as follows: After the limiting process $X^y$ (starting in $y\geq 0$) hits $0$ for the first time, it behaves in distribution like the limiting process $X^0$ (starting in $0$).
\\
The following theorem covers the behavior of $g^y$ for $y \to \infty$ conditioned on the existence of a zero as well as the behavior of $\Gamma^y$ and $g^y$ for $y \to -\infty$:
\begin{thm}
\label{thm3}
\begin{enumerate}[(a)]
\item For $y \to \infty$, the conditional distribution $\p\big(\frac{g^y}{y^2} \in \cdot \,\abs\, g^y \gne 0\big)$ converges weakly to an inverse chi-squared distribution with Lebesgue density $$\R \to [0,\infty), \quad s \mapsto \1_{\{s \gne 0\}} \frac{1}{\sqrt{2 \pi s^3}} \e^{-\frac{1}{2s}}.$$ 
In particular, $g^y$ conditional on $\{g^y \gne 0\}$ diverges in distribution to $\infty$ for $y \to \infty$.
\item For $y \to -\infty$, the random variable $y^2(1-\Gamma^y)$ converges weakly to an exponential distribution with parameter $\frac{1}{2}$ while $y^2(1-g^y)$ converges weakly to a random variable $g'$ with distribution function given by $$\p(g' \le u) =
\begin{cases}
\displaystyle\int_0^\infty \dfrac{2z}{\sqrt{2\pi(2z-u)}} \e^{-z} \D z, \quad & u \le 0,\bigskip\\
1- \dfrac{1}{2}\e^{-\frac{u}{2}}, \quad & u \ge 0.
\end{cases}$$
In particular, $\Gamma^y$ and $g^y$ both converge in distribution to $1$ for $y \to -\infty$.
\end{enumerate}
\end{thm}

\begin{figure}[h]
\centering
\label{Fig1}
\begin{minipage}[h]{0.496\textwidth}
\label{gPlot}
\includegraphics[width=\textwidth]{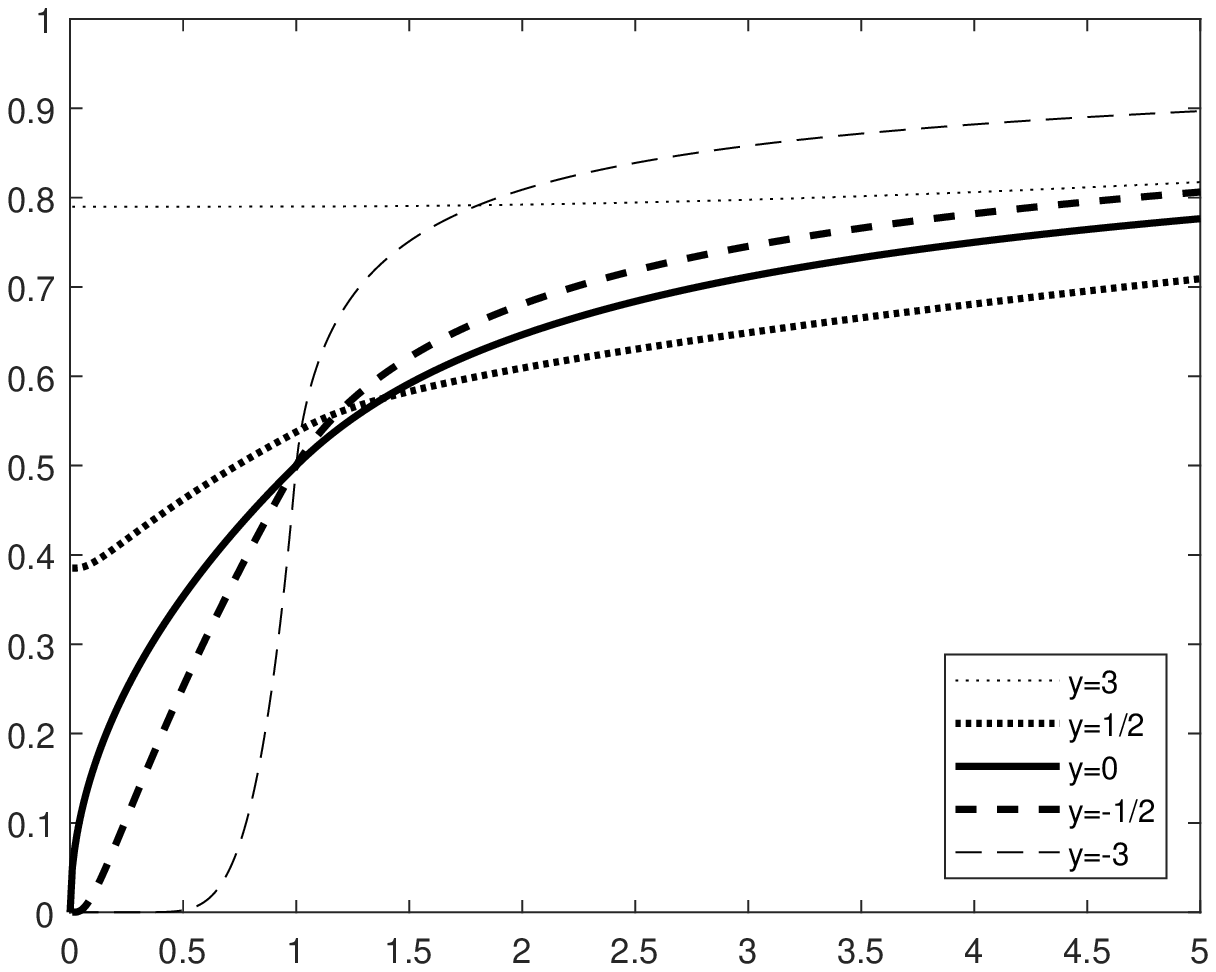}
\end{minipage}
\begin{minipage}[h]{0.496\textwidth}
\label{GammaPlot}
\includegraphics[width=\textwidth]{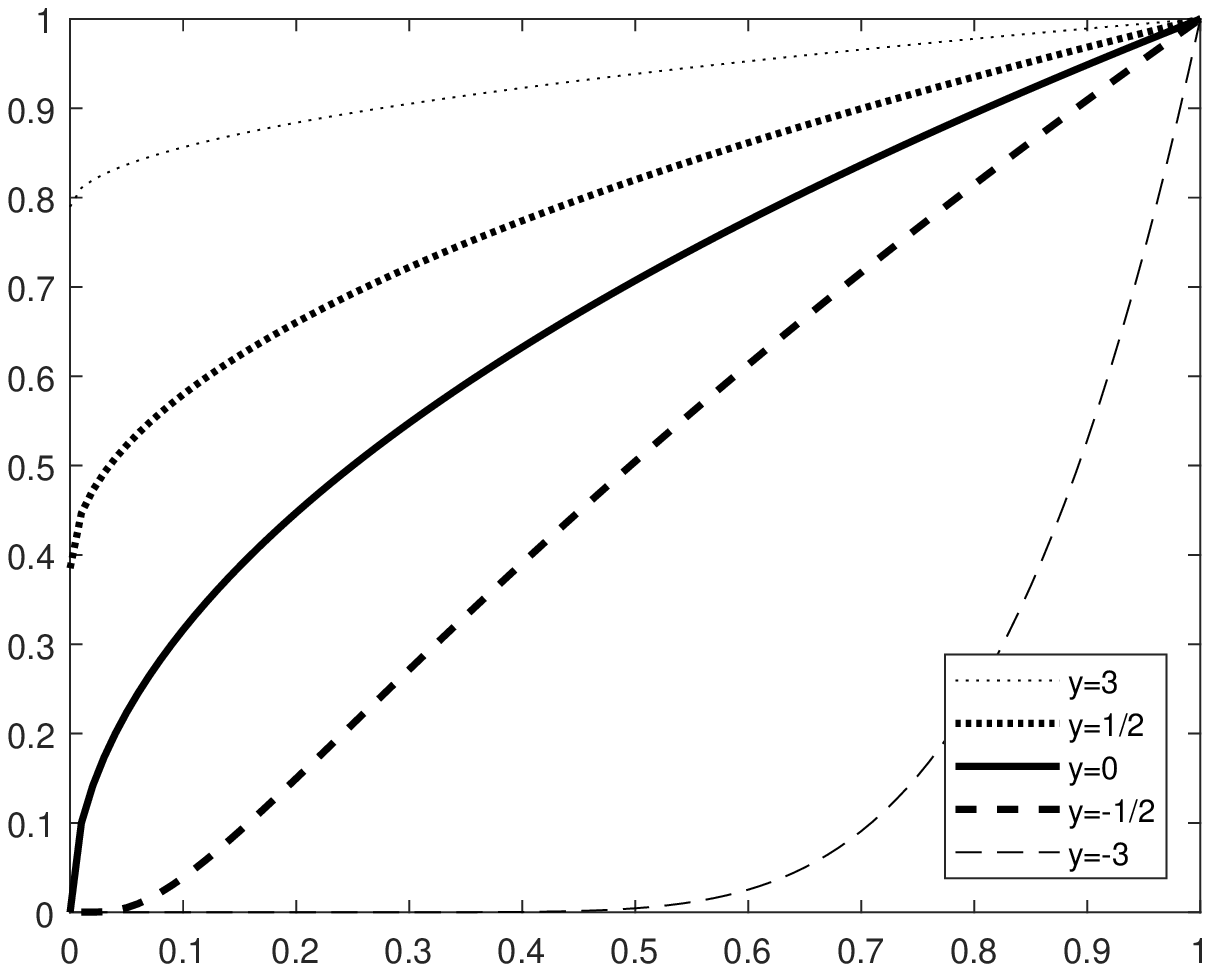}
\end{minipage}
\caption{Distribution functions of $g^y$ and of $\Gamma^y$, respectively, for several values of $y$}
\end{figure}

\section{Formulas for $q$ and Finiteness of $g$}
\label{Sec3}
Before we start proving Theorem~\ref{thm1}, we take a closer look at the distribution of $g$ given in~\eqref{gForm}. More precisely, we provide explicit formulas for $q$ (and hence for the distribution of $g$) and check that $g$ is finite almost surely. This is crucial for the limiting process $X$ to be well-defined. We start with an auxiliary result, which follows straight from a formula in \cite{BO99}:
\begin{lem}
\label{tauForm}
Given $y \neq 0$ as well as $T \gne 0$ and $z \in \R$, let $(b_t^z)_{t \in [0,T]}$ be a Brownian bridge of length $T$ with $b_0=y$ and $b_T=z$. We define $$\tau^z:= \min\{s \in [0,T]: b_s^z = 0\}$$ (with the convention $\min\emptyset:=T$). Then we have $$\p(\tau^z \in \D t) = \frac{\abs y\abs \sqrt{T}}{\sqrt{2\pi t^3(T-t)}}\e^{\frac{(y-z)^2}{2T}-\frac{z^2}{2(T-t)}-\frac{y^2}{2t}} \D t, \quad t \in (0,T).$$
This is a proper density integrating to~$1$ if and only if $zy \le 0$ holds.
\end{lem}
\begin{proof}
Let $(\bar{b}^z_t)_{t \in [0,T]}$ be a Brownian bridge of length $T$ with $\bar{b}^z_0=0$ and $\bar{b}^z_T= (y-z)\op{sgn}(y)$. Using a symmetry argument in the case $y \gne 0$, we get $$\tau^z \stackrel{d}{=} \min \{s \in [0,T]: \bar{b}_s^z = \abs y\abs\}.$$ The claim now immediately from formula~(2.15) of \cite{BO99}. Note that this formula is valid for all $\eta \in \R$ and not only for $\eta \lne \beta$ (compare Theorem~2.1 in \cite{BO99}).
\end{proof}
Combining this lemma with Lévy's result that the occupation time below~$0$ of a Brownian bridge without drift is uniformly distributed (see \cite{Lev40}), we compute the distribution function $q=q^y$ (see~\eqref{qDef}) of the occupation time of a Brownian bridge with drift starting in~$y$.
\begin{lem}
\label{qForm}
Given $t \gne 0$ and $u \in [0,t)$, we have
\begin{align}
\label{qFormEqn}
q(t,u) =
\begin{cases}
\dfrac{u}{t},& \quad y=0,\bigskip\\
\displaystyle\int_0^u \dfrac{\sqrt{t}(u-x)\abs y\abs}{\sqrt{2\pi x^3 (t-x)^3}} \e^{\frac{y^2}{2t}-\frac{y^2}{2x}} \D x,& \quad y\lne 0,\bigskip\\
\displaystyle\int_0^{t-u} \dfrac{\sqrt{t}uy}{\sqrt{2\pi x^3 (t-x)^3}} \e^{\frac{y^2}{2t}-\frac{y^2}{2x}} \D x+ \displaystyle\int_{t-u}^t \dfrac{\sqrt{t}y}{\sqrt{2\pi x^3 (t-x)}} \e^{\frac{y^2}{2t}-\frac{y^2}{2x}} \D x,& \quad y \gne 0.
\end{cases}
\end{align}
\end{lem}

After a linear time change, formula~(E-4) of \cite{Pec}, which is proved using Girsanov's theorem, provides the alternative representation
\begin{align}
\label{qFormPec}
q(t,u)
=
\begin{cases}
-2\left(\dfrac{t-u}{t}\left(1-\dfrac{y^2}{t}\right)-1\right) \Phi\left(\dfrac{y\sqrt{t-u}}{\sqrt{tu}}\right) +\dfrac{\sqrt{2u(t-u)}y}{\sqrt{\pi t^3}} \e^{\frac{y^2}{2t}-\frac{y^2}{2u}},& \quad y\le 0,\bigskip\\
1+ 2\left(\dfrac{u}{t}\left(1-\dfrac{y^2}{t}\right)-1\right) \Phi\left(-\dfrac{y\sqrt{u}}{\sqrt{t (t-u)}}\right)+\dfrac{\sqrt{2u(t-u)}y}{\sqrt{\pi t^3}} \e^{\frac{y^2}{2t}-\frac{y^2}{2(t-u)}},& \quad y\ge 0,
\end{cases}
\end{align}
for all $t \gne 0$ and $u \in [0,t)$, where $\Phi: \R \to \R$ denotes the distribution function of the standard normal distribution.
\begin{proof}[Proof of Lemma~\ref{qForm}.]
As already mentioned, the formula for $y=0$ is a classical result by Lévy (see \cite{Lev40}). Now let $y\neq 0$ and let $(b'_s)_{s \in [0,t]}$ be a Brownian bridge of length~$t$ with $b'_0=y$ and $b'_t=0$. Furthermore, let $\tau':= \min\{s \in [0,t]: b'_s=0\}$ be the first zero of $b'$.
Then $$\big(\hat{b}_s\big)_{s \in [0,1]}:= \left(\frac{1}{\sqrt{t-\tau'}}b'_{\tau' + s(t-\tau')}\right)_{s \in [0,1]}$$ is a standard Brownian bridge independent of $\tau'$.
By Lévy's result, $\int_0^1 \1_{\{\hat{b}_s \lne 0\}} \D s$ is uniformly distributed on $[0,1]$. Now let $y \lne 0$. We observe $b'_s \lne 0$ for all $s \lne \tau'$. Together with Lemma~\ref{tauForm}, we obtain
\begin{align*}
q(t,u)
=& \p\left(\int_{\tau'}^t \1_{\{b'_s \lne 0\}} \D s \le u-\tau' \right)
= \p\left(\int_0^1 \1_{\{\hat{b}_s \lne 0\}} \D s \le \frac{u-\tau'}{t-\tau'} \right)\\
=& \int_0^u \p\left(\int_0^1 \1_{\{\hat{b}_s \lne 0\}} \D s \le \frac{u-x}{t-x} \right) \p(\tau' \in \D x)
= \int_0^u \frac{u-x}{t-x}\cdot\frac{\sqrt{t}\abs y\abs}{\sqrt{2\pi x^3 (t-x)}} \e^{\frac{y^2}{2t}-\frac{y^2}{2x}} \D x.
\end{align*}
Given $y \gne 0$, we similarly observe $b'_s \gne 0$ for all $s \lne \tau'$ and obtain
\begin{align*}
q(t,u)
=& \p\left(\int_{\tau'}^t \1_{\{b'_s \lne 0\}} \D s \le u \right)
= \p\left(\int_0^1 \1_{\{\hat{b}_s \lne 0\}} \D s \le \frac{u}{t-\tau'} \right)\\
=& \int_0^t \p\left(\int_0^1 \1_{\{\hat{b}_s \lne 0\}} \D s \le \frac{u}{t-x} \right) \p(\tau' \in \D x)
= \int_0^t \left(\frac{u}{t-x}\wedge 1\right)\frac{\sqrt{t}y}{\sqrt{2\pi x^3 (t-x)}} \e^{\frac{y^2}{2t}-\frac{y^2}{2x}} \D x
\end{align*}
proving the claim.
\end{proof}

Next we prove that $g$ is almost surely finite. This will essentially follow from the subsequent lemma as we will see in Corollary \ref{gFinite} below.
The more general formulation of the lemma will help us to determine the distribution of $\Gamma$ (see Theorem~\ref{thm2}). Besides, it is rather difficult to verify the formulas directly for a single $u \gne 0$.

\begin{lem}
\label{IntLem}
We have
$$\int_u^\infty q(t,u) \frac{1}{\sqrt{t}} \e^{-\frac{y^2}{2t}} \D t = \int_0^u \frac{1}{\sqrt{t}} \e^{-\frac{y^2}{2t}} \D t, \quad u \ge 0,\ y \lne 0$$ and $$\int_0^u \frac{1}{\sqrt{t}} \e^{-\frac{y^2}{2t}} \D t + \int_u^\infty q(t,u) \frac{1}{\sqrt{t}} \e^{-\frac{y^2}{2t}} \D t = 4\sqrt{u}, \quad u \ge 0,\ y \gne 0.$$
\end{lem}
\begin{cor}
\label{gFinite}
The random variable $g$, as defined in~\eqref{gForm}, is almost surely finite.
\end{cor}
\begin{proof}
For $y=0$, the claim is clear from the definition. Applying Lemma~\ref{IntLem} with $u=1$ in the final step, we obtain
\begin{align*}
\p(g\lne \infty) \stackrel{\eqref{gForm}}{=} \dfrac{\int_0^\infty q(t,1) \frac{1}{\sqrt{t}}\e^{-\frac{y^2}{2t}} \D t}{2\int_0^1 \frac{1}{\sqrt{t}}\e^{-\frac{y^2}{2t}} \D t} \stackrel{\eqref{q=1}}{=} \dfrac{\int_0^1 \frac{1}{\sqrt{t}}\e^{-\frac{y^2}{2t}} \D t + \int_1^\infty q(t,1) \frac{1}{\sqrt{t}}\e^{-\frac{y^2}{2t}} \D t}{2\int_0^1 \frac{1}{\sqrt{t}}\e^{-\frac{y^2}{2t}} \D t}
=1, \quad y \lne 0.
\end{align*}
Similarly, we get
\begin{align*}
\p(g\lne \infty)
=\dfrac{2 \sqrt{2\pi} y + \int_0^1 \frac{1}{\sqrt{t}}\e^{-\frac{y^2}{2t}} \D t + \int_1^\infty q(t,1) \frac{1}{\sqrt{t}}\e^{-\frac{y^2}{2t}} \D t}{2 \sqrt{2\pi} y+4}
=1, \quad y \gne 0,
\end{align*}
proving the claim.
\end{proof}
\begin{proof}[Proof of Lemma~\ref{IntLem}.]
Noting $q(t,0)=0$ for all $t \gne 0$, both formulas hold for $u=0$. Hence it suffices to show that the derivatives of the left and the right-hand side coincide in both cases.\\
First we consider $y \lne 0$. Given $t \ge 0$, an application of Leibniz's rule yields $$\frac{\D}{\D u}  q(t,u) \frac{1}{\sqrt{t}} \e^{-\frac{y^2}{2t}} \stackrel{\eqref{qFormEqn}}{=} \frac{\D}{\D u} \int_0^u\frac{(u-x)\abs y\abs}{\sqrt{2\pi x^3 (t-x)^3}} \e^{-\frac{y^2}{2x}}\D x = \int_0^u\frac{\abs y\abs}{\sqrt{2\pi x^3 (t-x)^3}} \e^{-\frac{y^2}{2x}} \D x$$ for all $u \in (0,t)$. Now fix $u_0 \gne 0$. Noting
\begin{align*}
\int_{u_0}^\infty \sup_{u \in (0,u_0)} \left\abs \frac{\D}{\D u}  q(t,u) \frac{1}{\sqrt{t}} \e^{-\frac{y^2}{2t}}\right\abs \D t
=& \int_{u_0}^\infty \sup_{u \in (0,u_0)} \left\abs \int_0^u\frac{\abs y\abs}{\sqrt{2\pi x^3 (t-x)^3}} \e^{-\frac{y^2}{2x}} \D x\right\abs \D t\\
=& \int_{u_0}^\infty\int_0^{u_0}\frac{\abs y\abs}{\sqrt{2\pi x^3 (t-x)^3}} \e^{-\frac{y^2}{2x}} \D x\D t\\
=& \int_0^{u_0} \frac{2\abs y\abs}{\sqrt{2\pi x^3 (u_0-x)}} \e^{-\frac{y^2}{2x}} \D x\\
\lne& \infty,
\end{align*}
we can differentiate w.r.t.~$u \in (0,u_0)$ under the integral $\int_{u_0}^\infty$. Together with Leibniz's rule applied to the integral $\int_u^{u_0}$, we obtain 
\begin{align*}
\frac{\D}{\D u} \int_u^\infty q(t,u) \frac{1}{\sqrt{t}} \e^{-\frac{y^2}{2t}} \D t
=& \int_u^\infty \frac{\D}{\D u}  q(t,u) \frac{1}{\sqrt{t}} \e^{-\frac{y^2}{2t}} \D t- q(u,u) \frac{1}{\sqrt{u}} \e^{-\frac{y^2}{2u}}\\
=& \int_u^\infty \int_0^u\frac{\abs y\abs}{\sqrt{2\pi x^3 (t-x)^3}} \e^{-\frac{y^2}{2x}} \D x \D t- \frac{1}{\sqrt{u}} \e^{-\frac{y^2}{2u}}\\
=& \int_0^u\frac{2\abs y\abs}{\sqrt{2\pi x^3 (u-x)}} \e^{-\frac{y^2}{2x}} \D x- \frac{1}{\sqrt{u}} \e^{-\frac{y^2}{2u}}\\
=& \int_0^u \frac{\abs y\abs\sqrt{u}}{\sqrt{2\pi x^3 (u-x)}} \e^{\frac{y^2}{2u}-\frac{y^2}{2x}} \D x\cdot \frac{2}{\sqrt{u}}\e^{-\frac{y^2}{2u}}- \frac{1}{\sqrt{u}} \e^{-\frac{y^2}{2u}}
\end{align*}
for all $u \in (0,u_0)$ and consequently for all $u \gne 0$.
Given a Brownian bridge $(b'_{s,u})_{s \in [0,u]}$ of length~$u$ with $b'_{0,u}=y$ and $b'_{u,u}=0$, Lemma~\ref{tauForm} implies $$1 = \p(\min\{s \in [0,u]: b_{s,u}'=0\} \in (0,u)) = \int_0^u\frac{\abs y\abs\sqrt{u}}{\sqrt{2\pi x^3 (u-x)}} \e^{\frac{y^2}{2u}-\frac{y^2}{2x}} \D x$$ so that we can deduce
\begin{align*}
\frac{\D}{\D u} \int_u^\infty q(t,u) \frac{1}{\sqrt{t}} \e^{-\frac{y^2}{2t}} \D t
=& 1\cdot \frac{2}{\sqrt{u}}\e^{-\frac{y^2}{2u}}- \frac{1}{\sqrt{u}} \e^{-\frac{y^2}{2u}}
= \frac{1}{\sqrt{u}} \e^{-\frac{y^2}{2u}} = \frac{\D}{\D u} \int_0^u \frac{1}{\sqrt{t}} \e^{-\frac{y^2}{2t}} \D t
\end{align*}
for all $u \gne 0$, as claimed.
\\
Now we consider $y \gne 0$. Given $t \ge 0$, an application of Leibniz's rule yields
\begin{align}
\label{dduq}
\frac{\D}{\D u} q(t,u)  \frac{1}{\sqrt{t}} \e^{-\frac{y^2}{2t}}
\stackrel{\eqref{qFormEqn}}{=}& \frac{\D}{\D u} \int_0^{t-u} \frac{uy}{\sqrt{2\pi x^3 (t-x)^3}} \e^{-\frac{y^2}{2x}} \D x+ \frac{\D}{\D u} \int_{t-u}^t \frac{y}{\sqrt{2\pi x^3 (t-x)}} \e^{-\frac{y^2}{2x}} \D x\notag\\
=&\int_0^{t-u} \frac{y}{\sqrt{2\pi x^3 (t-x)^3}} \e^{-\frac{y^2}{2x}} \D x -\frac{uy}{\sqrt{2\pi (t-u)^3 u^3}} \e^{-\frac{y^2}{2(t-u)}} + \frac{y}{\sqrt{2\pi (t-u)^3 u}} \e^{-\frac{y^2}{2(t-u)}}\notag\\
=& \int_0^{t-u} \frac{y}{\sqrt{2\pi x^3 (t-x)^3}} \e^{-\frac{y^2}{2x}} \D x
\end{align}
for all $u \in (0,t)$. Now fix $u_0, \eps \gne 0$. We get
\begin{align*}
\int_{u_0}^\infty \sup_{u \in (\eps,u_0)} \left\abs \frac{\D}{\D u}  q(t,u) \frac{1}{\sqrt{t}} \e^{-\frac{y^2}{2t}}\right\abs \D t
=& \int_{u_0}^\infty \sup_{u \in (\eps,u_0)} \left\abs \int_0^{t-u} \frac{y}{\sqrt{2\pi x^3 (t-x)^3}} \e^{-\frac{y^2}{2x}} \D x\right\abs \D t\\
=& \int_{u_0}^\infty\int_0^{t-\eps}\frac{y}{\sqrt{2\pi x^3 (t-x)^3}} \e^{-\frac{y^2}{2x}} \D x\D t\\
=& \int_0^\infty\int_{u_0\vee (x+\eps)}^\infty\frac{y}{\sqrt{2\pi x^3 (t-x)^3}} \e^{-\frac{y^2}{2x}} \D t\D x\\
=& \int_0^\infty \frac{2y}{\sqrt{2\pi x^3 ((u_0-x)\vee \eps)}} \e^{-\frac{y^2}{2x}} \D x\\
\lne& \infty
\end{align*}
so that we can differentiate w.r.t.~$u \in (\eps,u_0)$ under the integral $\int_{u_0}^\infty$. Together with Leibniz's rule applied to the integral $\int_u^{u_0}$ and with $q(u,u)=1$, we obtain 
\begin{align*}
&\frac{\D}{\D u} \left(\int_0^u \frac{1}{\sqrt{t}} \e^{-\frac{y^2}{2t}} \D t + \int_u^\infty q(t,u) \frac{1}{\sqrt{t}} \e^{-\frac{y^2}{2t}} \D t\right)\\
\stackrel{\hphantom{\eqref{dduq}}}{=}& \frac{1}{\sqrt{u}} \e^{-\frac{y^2}{2u}} + \int_u^\infty \frac{\D}{\D u}  q(t,u) \frac{1}{\sqrt{t}} \e^{-\frac{y^2}{2t}} \D t- q(u,u) \frac{1}{\sqrt{u}} \e^{-\frac{y^2}{2u}}\\
\stackrel{\eqref{dduq}}{=}& \int_u^\infty \int_0^{t-u} \frac{y}{\sqrt{2\pi x^3 (t-x)^3}} \e^{-\frac{y^2}{2x}} \D x \D t
\end{align*}
for all $u \in (\eps,u_0)$ and consequently for all $u \gne 0$. Using the substitution $x=\frac{y^2}{z^2}$, we get
\begin{align*}
\int_u^\infty \int_0^{t-u} \frac{y}{\sqrt{2\pi x^3 (t-x)^3}} \e^{-\frac{y^2}{2x}} \D x \D t
=& \int_0^\infty \int_{x+u}^\infty \frac{y}{\sqrt{2\pi x^3 (t-x)^3}} \e^{-\frac{y^2}{2x}} \D t \D x
= \int_0^\infty\frac{2y}{\sqrt{2\pi x^3 u}} \e^{-\frac{y^2}{2x}} \D x\\
=& \int_0^\infty\frac{2y z^3}{\sqrt{2\pi y^6 u}} \e^{-\frac{z^2}{2}} \frac{2y^2}{z^3}\D z
= \frac{2}{\sqrt{u}} \int_0^\infty\frac{2}{\sqrt{2\pi}} \e^{-\frac{z^2}{2}}\D z
= \frac{2}{\sqrt{u}}
\end{align*}
proving $$\frac{\D}{\D u} \left(\int_0^u \frac{1}{\sqrt{t}} \e^{-\frac{y^2}{2t}} \D t + \int_u^\infty q(t,u) \frac{1}{\sqrt{t}} \e^{-\frac{y^2}{2t}} \D t\right) = \frac{\D}{\D u} 4 \sqrt{u}$$
for all $u \gne 0$.
\end{proof}

\section{Convergence of the Distribution of the Last Zero}
\label{Sec4}
Recall that $$g_T:= \max\{s \in [0,T]: B_s=0\}$$ denotes the last zero of $B$ before time $T \gne 0$ with the convention $\max\emptyset:=0$. As already mentioned in the introduction, the following result is the key part of the proof of Theorem~\ref{thm1}:
\begin{prop}
\label{gConv}
For $T \to \infty$, the probability measures $\p(g_T \in \cdot\, \,\abs\, \Gamma_T \le 1)$ converge to the law of $g$ in total variation.
\end{prop}

To prove this result, we compute the asymptotics of the condition $\p(\Gamma_T \le 1)$ and of the density $$\frac{\p(g_T \in \D x, \Gamma_T \le 1)}{\D x}, \quad x \gne 0,$$ for $T \to \infty$. After considering the mass in~$0$, corresponding to the case that the process has no zero, the proposition follows from Scheffé's lemma. Regarding the asymptotics of the mentioned density, the main idea is to condition on $g_T$.
To avoid a case distinction and obtain the rather compact formula~\eqref{gForm}, we additionally condition on $B_T$.

\begin{proof}[Proof of Proposition~\ref{gConv}.]
\quad\\
\textit{Step 1:}
We start by explicitly computing the density $\frac{\p(g_T \in \D x, \Gamma_T \le 1)}{\D x}$ and its asymptotics for $x \gne 0$. Let $x_0 \ge 0$ and $T \gne x_0+1$. Moreover, let $z \gne 0$ and let $b^z=(b^{z}_t)_{t \in [0,T]}$ be a Brownian bridge of length $T$ starting in~$y$ and ending in~$z$. We define $\gamma_T^{z}:= \max\{s \in [0,T]: b^{z}_s=0\}$. Conditioned on $\gamma_T^z \gne 0$ (i.e., on the existence of a zero of $b^z$), the process $$\big(\hat{b}^{z}_s\big)_{s \in [0,1]}:= \left(\frac{1}{\sqrt{\gamma_T^{z}}}\big(b^{z}_{s\gamma_T^{z}}-y+s y\big)\right)_{s \in [0,1]}$$ is a standard Brownian bridge independent of $\gamma_T^{z}$.
Hence $(\sqrt{t}\hat{b}^{z}_{\frac{s}{t}}+y-\frac{s}{t}y)_{s \in [0,t]}$ is a Brownian bridge of length~$t$ starting in~$y$ and ending in~$0$ for each $t \gne 0$. Using $T \gne x_0+1$ in the first step, we get
\begin{align}
\label{**}
\p(g_T \in (0,x_0], \Gamma_T \le 1)
=& \p\big(g_T \in (0,x_0], \Gamma_{g_T} \le 1, B_T \gne 0\big)\notag\\
=& \int_0^\infty \p\left(\gamma_T^{z} \in (0,x_0], \int_0^{\gamma_T^{z}} \1_{\{b_s^{z} \lne 0\}} \D s \le 1\right) \p(B_T \in \D z)\notag\\
=& \int_0^\infty \p\left(\gamma_T^{z} \in (0,x_0], \gamma_T^{z} \int_0^{1} \1_{\big\{b_{s\gamma_T^{z}}^{z} \lne 0\big\}} \D s \le 1\right) \p(B_T \in \D z)\notag\\
=& \int_0^\infty \p\left(\gamma_T^{z} \in (0,x_0], \gamma_T^{z}\int_0^{1} \1_{\big\{\sqrt{\gamma_T^{z}}\hat{b}^{z}_s+y-sy \lne 0\big\}} \D s \le 1\right) \p(B_T \in \D z)\notag\\
=& \int_0^\infty \int_{(0,x_0]} \p\left(t\int_0^{1} \1_{\big\{\sqrt{t}\hat{b}^{z}_s+y-sy \lne 0\big\}} \D s \le 1\right) \p(\gamma_T^{z} \in \D t) \p(B_T \in \D z)\notag\\
=& \int_0^\infty \int_{(0,x_0]} \p\left(\int_0^t \1_{\big\{\sqrt{t}\hat{b}^{z}_{\frac{s}{t}}+y-\frac{s}{t}y \lne 0\big\}} \D s \le 1\right) \p(\gamma_T^{z} \in \D t) \p(B_T \in \D z)\notag\\
=& \int_0^\infty \int_{(0,x_0]} q(t,1) \p(\gamma_T^{z} \in \D t) \p(B_T \in \D z).
\end{align}
Let $\big(\bar{b}_s^z\big)_{s \in [0,T]}$ be a Brownian bridge of length $T$ starting in~$z$ and ending in~$y$. According to Lemma~\ref{tauForm}, we have $$\p(\min\{s \in [0,T]:\bar{b}_s^z=0\} \in \D t) = \frac{z\sqrt{T}}{\sqrt{2\pi t^3(T-t)}}\e^{\frac{(z-y)^2}{2T}-\frac{y^2}{2(T-t)}-\frac{z^2}{2t}} \D t, \quad t \in (0,T)$$ and hence
$$\p(\gamma_T^{z} \in \D t) = \p(T-\min\{s \in [0,T]: \bar{b}_s^z=0\} \in \D t) = \frac{z\sqrt{T}}{\sqrt{2\pi (T-t)^3t}}\e^{\frac{(z-y)^2}{2T}-\frac{y^2}{2t}-\frac{z^2}{2(T-t)}} \D t, \quad t \in (0,T).$$ Combining this with~\eqref{**}, we obtain
\begin{align*}
\p(g_T \in (0,x_0], \Gamma_T \le 1)
=& \int_0^\infty \int_0^{x_0} q(t,1) \frac{z\sqrt{T}}{\sqrt{2\pi (T-t)^3t}}\e^{\frac{(z-y)^2}{2T}-\frac{y^2}{2t}-\frac{z^2}{2(T-t)}} \D t \frac{1}{\sqrt{2\pi T}} \e^{-\frac{(z-y)^2}{2T}} \D z\\
=& \frac{1}{2\pi}\int_0^{x_0} q(t,1) \e^{-\frac{y^2}{2t}} \int_0^\infty \frac{z}{\sqrt{t(T-t)^3}}\e^{-\frac{z^2}{2(T-t)}} \D z \D t\\
=& \frac{1}{2\pi}\int_0^{x_0} q(t,1) \e^{-\frac{y^2}{2t}} \frac{1}{\sqrt{t(T-t)}} \D t\\
=& \frac{1}{2\pi}\int_0^{x_0} q(t,1) \frac{1}{\sqrt{t\left(1-\frac{t}{T}\right)}} \e^{-\frac{y^2}{2t}} \D t\cdot \frac{1}{\sqrt{T}}.
\end{align*}
Since $x_0 \gne 0$ has been chosen arbitrarily,
we can deduce
\begin{align}
\label{AsympDens}
\frac{\p(g_T \in \D x, \Gamma_T \le 1)}{\D x}
=& \frac{1}{2\pi} q(x,1) \frac{1}{\sqrt{x\left(1-\frac{x}{T}\right)}} \e^{-\frac{y^2}{2x}} \cdot \frac{1}{\sqrt{T}} \notag\\
\sim&
\frac{1}{2\pi} q(x,1)
\frac{1}{\sqrt{x}} \e^{-\frac{y^2}{2x}} \cdot \frac{1}{\sqrt{T}}, \quad T \to \infty,\ x \gne 0.
\end{align}
\\
\textit{Step 2:} Next we compute the asymptotics of $\p(\Gamma_T \le 1)$ and, combining it with~\eqref{AsympDens}, the limit of $\frac{\p(g_T \in \D x\,\abs\, \Gamma_T \le 1)}{\D x}$. First we consider $y=0$.
Using the scaling property of $B$, Lévy's arcsine law and the definition of the derivative, we get $$\p(\Gamma_T \le 1) = \p\left(\int_0^1 \1_{\left\{B_s \lne 0\right\}} \D s \le \frac{1}{T}\right)= \frac{2}{\pi}\arcsin\left(\frac{1}{\sqrt{T}}\right) \sim \frac{2}{\pi} \arcsin'(0)\frac{1}{\sqrt{T}} = \frac{2}{\pi} \cdot \frac{1}{\sqrt{T}}, \quad T \to \infty.$$ Combining this with~\eqref{AsympDens} and using~\eqref{q=1}~and~\eqref{qFormEqn} in the second step, we deduce
\begin{align*}
\frac{\p(g_T \in \D x\,\abs\, \Gamma_T \le 1)}{\D x}
\to \frac{1}{4} q(x,1) \frac{1}{\sqrt{x}}
\stackrel{\hphantom{\eqref{gForm}}}{=}& \1_{\{x \le 1\}} \frac{1}{4\sqrt{x}}+ \1_{\{x \gne 1\}}\frac{1}{4\sqrt{x^3}}\\
\stackrel{\eqref{gForm}}{=}& \frac{\p(g \in \D x)}{\D x}, \quad T \to \infty, \ x \gne 0, \ y=0.
\end{align*}
Now let $y \lne 0$ and let $B^0=(B_t^0)_{t \ge 0}$ be a Brownian motion starting in~$0$. Using the scaling property and the symmetry of $B^0$, we obtain
\begin{align*}
\p(\Gamma_T \le 1)
=& \p\left(\int_0^T \1_{\{B_s^0 \lne -y\}} \D s \le 1\right)
= \p\left(\int_0^1 \1_{\left\{B_s^0 \lne -\frac{y}{\sqrt{T}}\right\}} \D s \le \frac{1}{T}\right)\\
=& \p\left(\int_0^1 \1_{\left\{B_s^0 \gne -\frac{y}{\sqrt{T}}\right\}} \D s \ge 1-\frac{1}{T}\right)
= 1-\p\left(\int_0^1 \1_{\left\{B_s^0 \gne -\frac{y}{\sqrt{T}}\right\}} \D s \le 1-\frac{1}{T}\right).
\end{align*}
Applying formula~(12) of \cite{Tak98} and the substitution $z=\frac{t}{T}$, we deduce
\begin{align*}
\p(\Gamma_T \le 1)
=& \frac{1}{\pi} \int_0^{\frac{1}{T}} \frac{1}{\sqrt{z(1-z)}} \e^{-\frac{y^2}{2zT}} \D z
= \frac{1}{\pi} \int_0^1 \frac{1}{\sqrt{t(1-\frac{t}{T})}} \e^{-\frac{y^2}{2t}} \D t \cdot \frac{1}{\sqrt{T}}.
\end{align*}
The dominated convergence theorem yields
\begin{equation*}
\p(\Gamma_T \le 1) \sim \frac{1}{\pi} \int_0^1 \frac{1}{\sqrt{t}} \e^{-\frac{y^2}{2t}} \D t \cdot \frac{1}{\sqrt{T}}, \quad T \to \infty.
\end{equation*}
Combining this with~\eqref{AsympDens}, we obtain
\begin{align*}
\frac{\p(g_T \in \D x\,\abs\, \Gamma_T \le 1)}{\D x}
\to& \frac{q(x,1) \frac{1}{\sqrt{x}} \e^{-\frac{y^2}{2x}}}{2 \int_0^1 \frac{1}{\sqrt{t}} \e^{-\frac{y^2}{2t}} \D t}
\stackrel{\eqref{gForm}}{=} \frac{\p(g \in \D x)}{\D x}, \quad T \to \infty, \ x \gne 0, \ y \lne 0.
\end{align*}
Finally, we consider $y \gne 0$. Since $B^0$ is symmetric, we have
\begin{align*}
\p(\Gamma_T \le 1)
=& \p\left(\int_0^T \1_{\{B_s^0 \lne -y\}} \D s \le 1\right)
= \p\left(\int_0^1 \1_{\left\{B_s^0 \lne -\frac{y}{\sqrt{T}}\right\}} \D s \le \frac{1}{T}\right)
= \p\left(\int_0^1 \1_{\left\{B_s^0 \gne \frac{y}{\sqrt{T}}\right\}} \D s \le \frac{1}{T}\right)
\end{align*}
Using formula~(12) of \cite{Tak98} and the density of the arcsine distribution, we deduce
\begin{align}
\label{EnumPos1}
\p(\Gamma_T \le 1)
=& 1-\frac{1}{\pi} \int_0^{1-\frac{1}{T}} \frac{1}{\sqrt{z(1-z)}} \e^{-\frac{y^2}{2zT}} \D z
= \frac{1}{\pi} \int_0^{1} \frac{1}{\sqrt{z(1-z)}} \D z - \frac{1}{\pi} \int_0^{1-\frac{1}{T}} \frac{1}{\sqrt{z(1-z)}} \e^{-\frac{y^2}{2zT}} \D z\notag\\
=& \frac{1}{\pi}\left(\int_0^{1-\frac{1}{T}} \frac{1}{\sqrt{z(1-z)}} \left(1-\e^{-\frac{y^2}{2zT}}\right) \D z + \int_{1-\frac{1}{T}}^{1} \frac{1}{\sqrt{z(1-z)}} \D z\right).
\end{align}
Recalling $\arcsin'(0)=1$, the definition of the derivative yields
\begin{equation}
\label{lim1}
\sqrt{T}\int_{1-\frac{1}{T}}^{1} \frac{1}{\sqrt{z(1-z)}} \D z = \sqrt{T}\int_0^{\frac{1}{T}} \frac{1}{\sqrt{(1-z)z}} \D z = 2 \sqrt{T}\arcsin\left(\frac{1}{\sqrt{T}}\right) \to 2, \quad T \to \infty.
\end{equation}
To obtain the asymptotics of the first term in~\eqref{EnumPos1}, we split the integral at $\frac{1}{2}$. Noting $$1-\e^{-\frac{y^2}{2zT}} \le 1- \e^{-\frac{y^2}{T}} \le \frac{y^2}{T}, \quad z \in \left[\frac{1}{2}, 1\right],$$ the dominated convergence theorem implies
\begin{equation}
\label{lim2}
\sqrt{T} \int_{\frac{1}{2}}^{1-\frac{1}{T}} \frac{1}{\sqrt{z(1-z)}} \left(1-\e^{-\frac{y^2}{2zT}}\right) \D z \to 0, \quad T \to \infty.
\end{equation}
Using the substitution $t=\frac{y^2}{x^2}$ and integration by parts, we get
\begin{align*}
\int_0^\infty \frac{1}{\sqrt{t}} \left(1-\e^{-\frac{y^2}{2t}}\right) \D t
=& \int_0^\infty \frac{2y}{x^2} \left(1-\e^{-\frac{x^2}{2}}\right)\D x
= \sqrt{2\pi} y.
\end{align*}
In particular, we have
\begin{align*}
\int_0^\infty \sup_{T \ge 1} \left\abs\frac{1}{\sqrt{t(1-\frac{t}{T})}} \left(1-\e^{-\frac{y^2}{2t}} \right)\1_{\left\{t \le \frac{T}{2}\right\}}\right\abs \D t
\le& \int_0^\infty \sup_{T \ge 1} \left\abs \frac{\sqrt{2}}{\sqrt{t}} \left(1-\e^{-\frac{y^2}{2t}}\right)\1_{\left\{t \le \frac{T}{2}\right\}}\right\abs \D t\\
=& \int_0^\infty \frac{\sqrt{2}}{\sqrt{t}} \left(1-\e^{-\frac{y^2}{2t}}\right) \D t
\lne \infty
\end{align*}
so that we can apply the dominated convergence theorem: Together with the substitution $z= \frac{t}{T}$, we get
\begin{align}
\label{lim3}
\sqrt{T}\int_0^{\frac{1}{2}} \frac{1}{\sqrt{z(1-z)}} \left(1-\e^{-\frac{y^2}{2zT}}\right) \D z
=& \int_0^{\frac{T}{2}} \frac{1}{\sqrt{t(1-\frac{t}{T})}} \left(1-\e^{-\frac{y^2}{2t}}\right) \D t \to \sqrt{2\pi} y, \quad T \to \infty.
\end{align}
In view of the three limits \eqref{lim1}, \eqref{lim2} and \eqref{lim3}, equation~\eqref{EnumPos1} implies
\begin{equation}
\label{AsympGammaPos}
\p(\Gamma_T \le 1)
\sim \frac{1}{\pi} (\sqrt{2\pi} y+2)\frac{1}{\sqrt{T}}, \quad T \to \infty.
\end{equation}
Combining this with~\eqref{AsympDens}, we can deduce
\begin{align*}
\frac{\p(g_T \in \D x\,\abs\, \Gamma_T \le 1)}{\D x}
\to& \frac{q(x,1) \frac{1}{\sqrt{x}} \e^{-\frac{y^2}{2x}}}{2\sqrt{2\pi} y+4}
\stackrel{\eqref{gForm}}{=} \frac{\p(g \in \D x)}{\D x}, \quad T \to \infty, \ x \gne 0, \ y \gne 0.
\end{align*}
\\
\textit{Step 3:} Based on the results of step 2 and taking account of the mass in~$0$ in the case $y \gne 0$, we finally prove the claimed weak convergence. As a consequence of $T \gne 1$, the reflection principle and the dominated convergence theorem, we have
\begin{align*}
\p(g_T =0, \Gamma_T \le 1) =& \p(B_t \gne 0 \text{ for all } t \in [0,T])
= \p(B_t^0 \lne y \text{ for all } t \in [0,T])\\
=& \p(\abs B_T^0\abs \lne y) =  \1_{\{y \gne 0\}}\frac{2}{\sqrt{2\pi }}\int_{0}^y \e^{-\frac{z^2}{2T}} \D z \cdot \frac{1}{\sqrt{T}}
\sim  \1_{\{y \gne 0\}}\frac{2}{\sqrt{2\pi}} y \cdot \frac{1}{\sqrt{T}}, \quad T \to \infty.\notag
\end{align*}
Combining this with~\eqref{AsympGammaPos}, we obtain
\begin{align}
\label{gT=0}
\p(g_T =0\,\abs\, \Gamma_T \le 1)
\to  \1_{\{y \gne 0\}}\frac{\sqrt{2\pi}y}{\sqrt{2\pi}y+2} \stackrel{\eqref{gForm}}{=} \p(g=0), \quad T \to \infty.
\end{align}
Recalling $\p(g \lne \infty) =1$, we deduce
\begin{align}
\int_{(0,\infty)} \frac{\p(g_T \in \D x\,\abs\, \Gamma_T \le 1)}{\D x} \D x = \p(g_T \gne 0\,\abs\, \Gamma_T \le 1)
\to  \p(g\gne 0) = \int_{(0,\infty)} \frac{\p(g \in \D x)}{\D x} \D x, \quad T \to \infty.\notag
\end{align}
Since we have already proved
\begin{align*}
\frac{\p(g_T \in \D x\,\abs\, \Gamma_T \le 1)}{\D x}
\to
\frac{\p(g \in \D x)}{\D x}, \quad T \to \infty, \ x \gne 0,
\end{align*}
Scheffé's lemma implies that the restriction of $\p(g_T \in \cdot \,\,\abs\, \Gamma_T \le 1)$ to $((0,\infty),\BB((0,\infty)))$ converges in total variation to the corresponding restriction of the law of $g$ for $T \to \infty$. Combining this with~\eqref{gT=0} yields the claim.
\end{proof}

\section{Proof of the Weak Convergence}
\label{Sec5}
To prove the weak convergence claimed in Theorem~\ref{thm1}, we will (additionally) condition on $g_T$, the last zero before $T$. In view of Proposition~\ref{gConv}, the following result guarantees that it essentially suffices to show the weak convergence of the conditioned process.
\begin{prop}
\label{IntConv}
Let $\mu_\infty,\mu_1,\mu_2,\dots$ be probability measures on a measurable space $(\Omega,\FF)$ such that $(\mu_n)_{n \in \N}$ converges to $\mu_\infty$ in total variation. Given a metric space $E$, let $\nu_\infty,\nu_1,\nu_2,\dots$ be Markov kernels from $(\Omega,\FF)$ to $(E,\BB(E))$ satisfying $\nu_n(t,\,\cdot\,) \Rightarrow \nu_\infty(t,\,\cdot\,)$ for $n \to \infty$ and $\mu_\infty$-almost every $t \in \Omega$. Then we have $$\int_{\Omega} \nu_n(t,\,\cdot\,) \mu_n(\D t) \Rightarrow \int_{\Omega} \nu_\infty(t,\,\cdot\,) \mu_\infty(\D t), \quad n \to \infty.$$
\end{prop}
\begin{proof}
Let $F \subseteq E$ be closed. For each $n \in \N$, we have
\begin{align*}
\int_{\Omega} \nu_n(t,F) \mu_n(\D t)
\le& \int_{\Omega}  \nu_n(t,F)\, \abs \mu_n-\mu_\infty\abs(\D t)+ \int_{\Omega} \nu_n(t,F) \mu_\infty(\D t)\\
\le& \abs \mu_\infty-\mu_n\abs(\Omega)+ \int_{\Omega} \nu_n(t,F) \mu_\infty(\D t),
\end{align*}
where $\abs \mu_\infty- \mu_n\abs$ denotes the variation of the signed measure $\mu_\infty- \mu_n$. Applying Fatou's lemma and the Portmanteau theorem, we deduce
\begin{align*}
\limsup_{n \to \infty} \int_{\Omega} \nu_n(t,F) \mu_n(\D t)
\le& \lim_{n \to \infty} \abs \mu_\infty-\mu_n\abs(\Omega)+ \limsup_{n \to \infty} \int_{\Omega} \nu_n(t,F) \mu_\infty(\D t)\\
\le& 0+\int_{\Omega} \limsup_{n \to \infty} \nu_n(t,F) \mu_\infty(\D t)
\le \int_{\Omega} \nu_\infty(t,F) \mu_\infty(\D t)
\end{align*}
so that the Portmanteau theorem yields the claim.
\end{proof}

Around $g_T$, we can decompose $(B_t)_{t \in [0,T]}$ into a scaled Brownian bridge with drift and, up to sign, a scaled Brownian meander.
Proposition~\ref{AZconv} below shows that a scaled Brownian meander of infinite length is nothing but a three-dimensional Bessel process starting in $0$. Recall that $Z$ is such a process. Here and in what follows, weak convergence of stochastic processes is always understood as weak convergence in the space of continuous functions with the topology of locally uniform convergence and the corresponding Borel $\sigma$-algebra.
\begin{prop}
\label{AZconv}
Let $t_0 \gne 0$ and $x \in [0,t_0]$. Moreover, let $(A_s)_{s \in [0,1]}$ be a Brownian meander. Then we have $$\left(\sqrt{T-x} A_{\frac{t}{T-x}}\right)_{t \in [0,t_0-x]} \Rightarrow (Z_t)_{t \in [0,t_0-x]}, \quad T \to \infty.$$
\end{prop}
\begin{proof}
Let $f \in \CC_b(\CC([0,t_0-x]))$ and $T \gne t_0$. Section~4 of \cite{Imh84} provides the change of measure formula
\begin{equation}
\label{AZtrafo}
\p\left((A_s)_{s \in [0,1]} \in \cdot \,\right) = \E\left(\frac{1}{Z_1} \sqrt{\frac{\pi}{2}}\1_{(Z_s)_{s \in [0,1]} \in \,\cdot\,} \right).
\end{equation}
Together with the scaling property of $Z$, we get
\begin{align*}
\E\left(f\left(\left(\sqrt{T-x} A_{\frac{t}{T-x}}\right)_{t \in [0,t_0-x]}\right)\right)
=& \E\left(\frac{1}{Z_1} \sqrt{\frac{\pi}{2}} f\left(\left(\sqrt{T-x} Z_{\frac{t}{T-x}}\right)_{t \in [0,t_0-x]}\right)\right)\\
=& \E\left(\frac{\sqrt{T-x}}{Z_{T-x}} \sqrt{\frac{\pi}{2}} f\left((Z_t)_{t \in [0,t_0-x]}\right)\right).
\end{align*}
Given the three-dimensional Brownian motion $(W_t)_{t \ge 0}$ with $Z= \y W\y_2$, the process $$Z':=(Z'_t)_{t \ge 0}:= (\y W_{t_0-x+t}-W_{t_0-x}\y_2)_{t \ge 0}$$ is a three-dimensional Bessel process independent of $(Z_t)_{t \in [0,t_0-x]}$. The triangle inequality implies $Z_{T-x} \le Z'_{T-t_0}+Z_{t_0-x}$. Combining this with the scaling property of $Z'$ and the dominated convergence theorem, we obtain
\begin{align*}
\E\left(\frac{\sqrt{T-x}}{Z_{T-x}} \sqrt{\frac{\pi}{2}} f\left((Z_t)_{t \in [0,t_0-x]}\right)\right)
\ge & \E\left(\frac{\sqrt{T-x}}{Z'_{T-t_0}+Z_{t_0-x}} \sqrt{\frac{\pi}{2}} f\left((Z_t)_{t \in [0,t_0-x]}\right)\right)\\
=& \E\left(\frac{\sqrt{T-x}}{\sqrt{T-t_0}Z'_{1}+Z_{t_0-x}} \sqrt{\frac{\pi}{2}} f\left((Z_t)_{t \in [0,t_0-x]}\right)\right)\\
=& \frac{\sqrt{T-x}}{\sqrt{T-t_0}} \E\left(\frac{1}{1+\frac{Z_{t_0-x}}{\sqrt{T-t_0}Z'_1}}\cdot \frac{1}{Z_1'}\sqrt{\frac{\pi}{2}} f\left((Z_t)_{t \in [0,t_0-x]}\right)\right)\\
\to& \E\left(\frac{1}{Z_1'}\sqrt{\frac{\pi}{2}} f\left((Z_t)_{t \in [0,t_0-x]}\right)\right), \quad T \to \infty.
\end{align*}
On the other hand, the triangle inequality also implies $Z_{T-x} \ge Z'_{T-t_0}-Z_{t_0-x}$. Together with $\lim_{T \to \infty} Z'_{T-t_0} = \infty$ a.s., a similar computation yields $$\limsup_{T \to \infty} \E\left(\1_{\{2Z_{t_0-x}\le Z'_{T-t_0}\}} \frac{\sqrt{T-x}}{Z_{T-x}} \sqrt{\frac{\pi}{2}} f\left((Z_t)_{t \in [0,t_0-x]}\right)\right) \le \E\left(\frac{1}{Z_1'}\sqrt{\frac{\pi}{2}} f\left((Z_t)_{t \in [0,t_0-x]}\right)\right).$$ Using $Z_{T-x} \le Z'_{T-t_0}+Z_{t_0-x}$ again, we get
\begin{align*}
0 \le& \E\left(\1_{\{2Z_{t_0-x}\gne Z'_{T-t_0}\}} \frac{\sqrt{T-x}}{Z_{T-x}} \sqrt{\frac{\pi}{2}} f\left((Z_t)_{t \in [0,t_0-x]}\right)\right)
\le \sqrt{\frac{\pi}{2}}\y f\y_\infty\E\left(\1_{\{3Z_{t_0-x}\gne Z_{T-x}\}} \frac{\sqrt{T-x}}{Z_{T-x}}\right)\\
=&  \sqrt{\frac{\pi}{2}}\y f\y_\infty \E\left(\1_{\Big\{3Z_{\frac{t_0-x}{T-x}} \gne Z_1\Big\}} \frac{1}{Z_{1}}\right)
\to  \sqrt{\frac{\pi}{2}}\y f\y_\infty\E\left(\1_{\{0 \gne Z_1\}} \frac{1}{Z_{1}}\right)
= 0, \quad T \to \infty.
\end{align*}
As a consequence of~\eqref{AZtrafo}, we have $\E\big(\frac{1}{Z_1'}\sqrt{\frac{\pi}{2}}\big)=1$. Using this fact as well as the independence of $Z'$ and $(Z_t)_{t \in [0,t_0-x]}$ in the final step, we can deduce
\begin{align*}
\E\left(f\left(\left(\sqrt{T-x} A_{\frac{t}{T-x}}\right)_{t \in [0,t_0-x]}\right)\right)
=& \E\left(\frac{\sqrt{T-x}}{Z_{T-x}} \sqrt{\frac{\pi}{2}} f\left((Z_t)_{t \in [0,t_0-x]}\right)\right)\\
\to& \E\left(\frac{1}{Z_1'}\sqrt{\frac{\pi}{2}} f\left((Z_t)_{t \in [0,t_0-x]}\right)\right)
= \E\left(f\left((Z_t)_{t \in [0,t_0-x]}\right)\right), \quad T \to \infty,
\end{align*}
proving the claim.
\end{proof}

Finally, we are ready to prove Theorem~\ref{thm1}. In view of the topological structure of $\CC([0,\infty))$, it suffices to prove weak convergence in $\CC([0,t_0])$ for each $t_0 \gne 0$. If there exists a zero before time $T \gne t_0$, we use the mentioned path decomposition
and the two auxiliary results we just proved. If there is no zero before $T$, we simply use the fact that Brownian motion conditioned to stay positive is nothing but a three-dimensional Bessel process.
\begin{proof}[Proof of Theorem~\ref{thm1}.]
Let $t_0 \gne 0$ and $T \gne t_0$. Conditioned on $g_T \gne 0$, $$(b_s')_{s \in [0,1]}:=\left(\frac{1}{\sqrt{g_T}} \big(B_{sg_T}-y+s y\big)\right)_{s \in [0,1]}$$ is a standard Brownian bridge and $$(A_s)_{s \in [0,1]}:=\left(\frac{1}{\sqrt{T-g_T}} \big\abs B_{g_T+(T-g_T)s}\big\abs \right)_{s \in [0,1]}$$ is a Brownian meander such that the two processes, $g_T$ and the sign of $B_T$ are mutually independent.
Rewriting the definitions, we get $$(B_t)_{t \in [0,T]} = \left(\1_{\{t \lne g_T\}} \left(\sqrt{g_T}b'_{\frac{t}{g_T}}+y-\frac{t}{g_T}y\right)+ \1_{\{t \ge g_T\}} \op{sgn}(B_T)\sqrt{T-g_T} A_{\frac{t-g_T}{T-g_T}} \right)_{t \in [0,T]}.$$ Now let $x \gne 0$ and $T \ge x+1$. Using this assumption in the first step, Proposition~\ref{AZconv} in the forth and the mentioned independence at several points, we obtain
\begin{align*}
&\p\left((B_t)_{t \in [0,t_0]} \in \cdot \,\, \big\abs\, \Gamma_T \le 1, g_T=x\right) \\
=& \p\left((B_t)_{t \in [0,t_0]} \in \cdot \,\, \big\abs\, \Gamma_x \le 1, B_T \gne 0, g_T=x\right)\\
=& \p\left(\left(\1_{\{t \lne x\}} \left(\sqrt{x} b'_{\frac{t}{x}}+y-\frac{t}{x}y\right) + \1_{\{t \ge x\}} \sqrt{T-x} A_{\frac{t-x}{T-x}} \right)_{t \in [0,t_0]} \in \cdot \,\, \right\abs\ldots\\
&\hspace{97pt}\ldots
\left\abs\, \int_0^{x} \1_{\left\{\sqrt{x} b'_{\frac{s}{x}}+y-\frac{s}{x}y \lne 0\right\}} \D s \le 1, B_T \gne 0, g_T=x \right)\\
=& \p\left(\left(\1_{\{t \lne x\}} \left(\sqrt{x} b'_{\frac{t}{x}}+y-\frac{t}{x}y\right) + \1_{\{t \ge x\}} \sqrt{T-x} A_{\frac{t-x}{T-x}} \right)_{t \in [0,t_0]} \in \cdot \,\, \Bigg\abs\, \int_0^{x} \1_{\left\{\sqrt{x} b'_{\frac{s}{x}}+y-\frac{s}{x}y \lne 0\right\}} \D s \le 1 \right)\\
\Rightarrow& \p\left(\left(\1_{\{t \lne x\}} \left(\sqrt{x} b'_{\frac{t}{x}}+y-\frac{t}{x}y\right)+ \1_{\{t \ge x\}} Z_{t-x} \right)_{t \in [0,t_0]} \in \cdot \,\, \Bigg\abs\, \int_0^{x} \1_{\left\{\sqrt{x} b'_{\frac{s}{x}}+y-\frac{s}{x}y \lne 0\right\}} \D s \le 1 \right)\\
=& \p\left(\left(\1_{\{t \lne x\}} \left(\sqrt{x}b_{\frac{t}{x}}+y-\frac{t}{x}y\right)+ \1_{\{t \ge x\}} Z_{t-x} \right)_{t \in [0,t_0]} \in \cdot \,\,\Bigg\abs\, g=x \right)\\
=& \p\left((X_t)_{t \in [0,t_0]} \in \cdot \,\,\big\abs\, g=x \right),
\quad T \to \infty.
\end{align*}
Given $y \gne 0$, we additionally have (see for instance Example 3 in \cite{Pin85})
\begin{align*}
\p\left((B_t)_{t \in [0,t_0]} \in \cdot \,\, \big\abs\, \Gamma_T \le 1, g_T=0\right)
&= \hspace{2.4pt}\p\left((B_t)_{t \in [0,t_0]} \in \cdot\, \, \big\abs\, B_t \gne 0 \text{ for all } t \in [0,T] \right)\\
&\Rightarrow \p\left((Y_t)_{t \in [0,t_0]} \in \cdot\,\right)
=  \p\big((X_t)_{t \in [0,t_0]} \in \cdot\,\big\abs \, g=0\big), \quad T \to \infty.
\end{align*}
In view of Proposition~\ref{gConv} and of $\p(g=0)=0$ for $y \le 0$, Proposition~\ref{IntConv} yields
\begin{align*}
\p\left((B_t)_{t \in [0,t_0]} \in \cdot \,\, \big\abs\, \Gamma_T \le 1\right)
&=\hspace{2.4pt} \int_{[0,\infty)} \p\left((B_t)_{t \in [0,t_0]} \in \cdot \,\, \big\abs\, \Gamma_T \le 1, g_T=x\right) \p(g_T \in \D x\,\abs\, \Gamma_T \le 1)\\
&\Rightarrow \int_{[0,\infty)} \p\left((X_t)_{t \in [0,t_0]} \in \cdot \,\,\big\abs\, g=x \right) \p(g \in \D x)\\
&=\hspace{2.4pt} \p\left((X_t)_{t \in [0,t_0]} \in \cdot \,\right), \quad T \to \infty.
\end{align*}
Since $t_0 \gne 0$ has been chosen arbitrarily, Theorem~5 of \cite{Whi70} yields the claimed weak convergence in $\CC([0,\infty))$.
\end{proof}

\section{Proofs of the Remaining Results}
\label{Sec6}
To prove Theorem~\ref{thm2}, we condition on $g$, the last zero of the limiting process, and perform an explicit calculation using Lemma~\ref{IntLem}.
\begin{proof}[Proof of Theorem~\ref{thm2}.]
Let $u \in [0,1]$. By construction of $X$, we have $$\Gamma = \int_0^{g} \1_{\{X_s \lne 0\}} \D s = \int_0^{g} \1_{\left\{\sqrt{g} b_{\frac{s}{g}}+y-\frac{s}{g}y \lne 0\right\}} \D s.$$ Now let $(b'_s)_{s \in [0,1]}$ be a standard Brownian bridge. Using $u \le 1$ and the definition of $b$ in the second step, we obtain
\begin{align*}
\p(\Gamma \le u)
=& \p(g\le u)+\int_u^\infty \p\left(\int_0^{g} \1_{\left\{\sqrt{g} b_{\frac{s}{g}}+y-\frac{s}{g}y \lne 0\right\}} \D s \le u\,\bigg\abs\, g=x\right) \p(g \in \D x)\\
=& \p(g\le u)+ \int_u^\infty \frac{\p\Big(\int_0^{x} \1_{\big\{\sqrt{x} b'_{\frac{s}{x}}+y-\frac{s}{x}y \lne 0\big\}} \D s \le u\Big)}{\p\Big(\int_0^{x} \1_{\big\{\sqrt{x} b'_{\frac{s}{x}}+y-\frac{s}{x}y \lne 0\big\}} \D s \le 1\Big)} \p( g \in \D x)\\
=& \p(g\le u)+\int_u^\infty \frac{q(x,u)}{q(x,1)} \p(g \in \D x).
\end{align*}
Applying Lemma~\ref{IntLem} in the final step, we deduce
\begin{align*}
\p(\Gamma \le u)
\stackrel{\eqref{gForm}}{=}& \frac{\int_0^u q(x,1)\frac{1}{\sqrt{x}} \e^{-\frac{y^2}{2x}} \D x}{2\int_0^1 \frac{1}{\sqrt{x}} \e^{-\frac{y^2}{2x}} \D x} + \frac{\int_u^\infty q(x,u)\frac{1}{\sqrt{x}} \e^{-\frac{y^2}{2x}} \D x}{2\int_0^1 \frac{1}{\sqrt{x}} \e^{-\frac{y^2}{2x}} \D x}\\
\stackrel{\eqref{q=1}}{=}& \frac{\int_0^u \frac{1}{\sqrt{x}} \e^{-\frac{y^2}{2x}} \D x}{2\int_0^1 \frac{1}{\sqrt{x}} \e^{-\frac{y^2}{2x}} \D x} + \frac{\int_u^\infty q(x,u)\frac{1}{\sqrt{x}} \e^{-\frac{y^2}{2x}} \D x}{2\int_0^1 \frac{1}{\sqrt{x}} \e^{-\frac{y^2}{2x}} \D x}
= 2\frac{\int_0^u \frac{1}{\sqrt{x}} \e^{-\frac{y^2}{2x}} \D x}{2\int_0^1 \frac{1}{\sqrt{x}} \e^{-\frac{y^2}{2x}} \D x}, \quad y \lne 0.
\end{align*}
Similarly, we get
\begin{align*}
\p(\Gamma \le u)
\stackrel{\eqref{gForm}}{=}& \frac{2\sqrt{2\pi}y+\int_0^u q(x,1)\frac{1}{\sqrt{x}} \e^{-\frac{y^2}{2x}} \D x}{2\sqrt{2\pi}y+4} + \frac{\int_u^\infty q(x,u)\frac{1}{\sqrt{x}} \e^{-\frac{y^2}{2x}} \D x}{2\sqrt{2\pi}y+4}\\
\stackrel{\eqref{q=1}}{=}&  \frac{2\sqrt{2\pi}y+\int_0^u \frac{1}{\sqrt{x}} \e^{-\frac{y^2}{2x}} \D x + \int_u^\infty q(x,u)\frac{1}{\sqrt{x}} \e^{-\frac{y^2}{2x}} \D x}{2\sqrt{2\pi}y+4}
= \frac{2\sqrt{2\pi}y+4\sqrt{u}}{2\sqrt{2\pi}y+4}, \quad y \gne 0.
\end{align*}
Given $y = 0$, we can proceed as in the proof of Theorem~4 in \cite{BB11}: Applying the explicit formulas~\eqref{q=1}~and~\eqref{qFormEqn} in the second step, we obtain
\begin{align*}
\p(\Gamma \le u)
\stackrel{\eqref{gForm}}{=}& \frac{\sqrt{u}}{2} + \frac{\int_u^1 \frac{q(x,u)}{q(x,1)\sqrt{x}}\D x}{4}+ \frac{\int_1^\infty \frac{q(x,u)}{q(x,1)\sqrt{x^3}} \D x}{4}
\stackrel{\hphantom{\eqref{gForm}}}{=} \frac{\sqrt{u}}{2} + \frac{\int_u^\infty \frac{u}{\sqrt{x^3}} \D x}{4}
= \sqrt{u}
= \frac{\int_0^u \frac{1}{\sqrt{x}}\D x}{\int_0^1 \frac{1}{\sqrt{x}}\D x}, \quad y=0,
\end{align*}
showing the claim.
\end{proof}

As already mentioned, we prove Proposition \ref{gGamma} without relying explicitly on the distributions of $g$, $\tau$ and $\Gamma$ but on the arcsine laws and the strong Markov property.
\begin{proof}[Proof of Proposition~\ref{gGamma}.]
Let $u \in [0,1]$ and $T \gne 2$. Moreover, let $$\bar{\tau}:= \inf\{t \ge 0: B_t = 0\}$$ be the first zero of $B$. According to the strong Markov property, $(\bar{B}_t)_{t \ge 0}:= (B_{\bar{\tau}+t})_{t \ge 0}$ is a standard Brownian motion independent of $\bar{\tau}$. For each $t \ge 0$, we define $$\bar{g}_t:= \sup\{s \in [0,t]: \bar{B}_s = 0\} \quad \text{and} \quad \bar{\Gamma}_t:= \int_0^t \1_{\{\bar{B}_s \lne 0\}} \D s.$$ On $\{\bar{\tau} \le T\}$, we have
\begin{align}
\label{gGammabar}
g_T = \bar{\tau} + \bar{g}_{T-\bar{\tau}} \quad \text{and} \quad \Gamma_T = \1_{\{y \le 0\}} \bar{\tau} + \bar{\Gamma}_{T -\bar{\tau}} = 
\begin{cases}
\bar{\tau} + \bar{\Gamma}_{T -\bar{\tau}}, \quad& y \le 0,\\
\bar{\Gamma}_{T -\bar{\tau}}, \quad& y \ge 0.
\end{cases}
\end{align}
Using the symmetry of $\bar{B}$ in the first step and the arcsine laws in the second, we obtain
\begin{align*}
2\p(\bar{g}_{T-t} \le u-t, \bar{B}_{T-t} \gne 0) = \p(\bar{g}_{T-t} \le u-t) = \p(\bar{\Gamma}_{T-t} \le u-t), \quad t \in [0,u].
\end{align*}
Recalling $T-u \gne 1$ and that $\bar{\tau}$ is independent of both $\bar{g}$ and $\bar{\Gamma}$, we can deduce
\begin{align*}
2\p(g_T \in (0,u], \Gamma_T \le 1)
\stackrel{\hphantom{\eqref{gGammabar}}}{=}& 2\p(g_T \in (0,u], B_T \gne 0)\\
\stackrel{\hphantom{\eqref{gGammabar}}}{=}& 2\p(g_T \le u, B_T \gne 0, \bar{\tau} \le u)\\
\stackrel{\eqref{gGammabar}}{=}& 2\p(\bar{g}_{T-\bar{\tau}} \le u-\bar{\tau}, \bar{B}_{T-\bar{\tau}} \gne 0, \bar{\tau} \le u)\\
\stackrel{\hphantom{\eqref{gGammabar}}}{=}& \int_{[0,u]} 2\p(\bar{g}_{T-t} \le u-t, \bar{B}_{T-t} \gne 0) \p(\bar{\tau} \in \D t)\\
\stackrel{\hphantom{\eqref{gGammabar}}}{=}& \int_{[0,u]} \p(\bar{\Gamma}_{T-t} \le u-t) \p(\bar{\tau} \in \D t)\\
\stackrel{\hphantom{\eqref{gGammabar}}}{=}& \p(\bar{\Gamma}_{T-\bar{\tau}} \le u-\bar{\tau}, \bar{\tau} \le u).
\end{align*}
Noting $\{\Gamma_T \le u\} \subseteq \{\bar{\tau} \le u\}$ for $y \le 0$, we get
\begin{align*}
2\p(g_T \in (0,u], \Gamma_T \le 1)
=&\p(\bar{\Gamma}_{T-\bar{\tau}} \le u-\bar{\tau}, \bar{\tau} \le u)
\stackrel{\eqref{gGammabar}}{=} \p\left(\1_{\{y \gne 0\}} \bar{\tau}+ \Gamma_T \le u, \bar{\tau} \le u\right)\\
=& \p\left(\1_{\{y \gne 0\}} \bar{\tau}+\Gamma_T \le u\right)
= \p\left(\1_{\{y \gne 0\}} \bar{\tau}+ \Gamma_T \le u, \Gamma_T \le 1\right)
\end{align*}
and hence
\begin{align}
\label{*}
2\p(g_T \in (0,u]\,\abs\, \Gamma_T \le 1)
= \p\left(\1_{\{y \gne 0\}} \bar{\tau}+\Gamma_T \le u\,\big\abs\, \Gamma_T \le 1\right).
\end{align}
As a consequence of Proposition~\ref{gConv}, the left-hand side converges to $2\p(g \in (0,u])$ for $T \to \infty$.
In view of the very same proposition, a path decomposition and conditioning argument similar to that in the proof of Theorem~\ref{thm1} implies that the right-hand side of~\eqref{*} converges to $$\p\big(\1_{\{y \gne 0\}}\inf\{t \ge 0: X_t = 0\} + \Gamma \le u\big) = \p(\tau + \Gamma \le u)$$ for $T \to \infty$.
\end{proof}

Theorem~\ref{thm3} follows from the formulas for the distribution functions of $g^y$ and $\Gamma^y$ by direct computation.
\begin{proof}[Proof of Theorem~\ref{thm3}.]
\textit{Part (a):} Let $u \gne 0$ and $y \gne \frac{2}{\sqrt{u}}$. Then
\begin{align}
\label{4aprf}
\p\left(\frac{g^y}{y^2} \gne u\,\bigg\abs\, g^y \gne 0\right)
\stackrel{\eqref{gForm}}{=}& \frac{1}{4}\int_{u y^2}^\infty q^y(t,1) \frac{1}{\sqrt{t}}\e^{-\frac{y^2}{2t}} \D t\notag\\
\stackrel{\eqref{qFormEqn}}{=}& \frac{1}{4}\int_{u y^2}^\infty \int_0^{t-1} \frac{y}{\sqrt{2\pi x^3 (t-x)^3}} \e^{-\frac{y^2}{2x}} \D x+ \int_{t-1}^t \frac{y}{\sqrt{2\pi x^3 (t-x)}} \e^{-\frac{y^2}{2x}} \D x \D t
\end{align}
holds. Regarding the first double integral, we use Fubini's theorem and two linear substitutions to obtain
\begin{align*}
&\int_{u y^2}^\infty \int_0^{t-1} \frac{y}{\sqrt{2\pi x^3 (t-x)^3}} \e^{-\frac{y^2}{2x}} \D x \D t\\
=& \int_{0}^\infty \int_{uy^2 \vee (x+1)}^\infty \frac{y}{\sqrt{2\pi x^3 (t-x)^3}} \e^{-\frac{y^2}{2x}} \D t \D x\\
=& \int_{0}^{uy^2-1} \frac{2y}{\sqrt{2\pi x^3 (uy^2-x)}} \e^{-\frac{y^2}{2x}} \D x + \int_{uy^2-1}^\infty \frac{2y}{\sqrt{2\pi x^3}} \e^{-\frac{y^2}{2x}} \D x\\
=& \int_{0}^{1} \frac{2y(uy^2-1)}{\sqrt{2\pi (uy^2-1)^3z^3 (uy^2-(uy^2-1)z)}} \e^{-\frac{y^2}{2(uy^2-1)z}} \D z
+ \int_{u}^\infty \frac{2y^3}{\sqrt{2\pi (sy^2-1)^3}} \e^{-\frac{y^2}{2(sy^2-1)}} \D s\\
=& \frac{1}{y}\int_{0}^{1} \frac{2\big(u-\frac{1}{y^2}\big)}{\sqrt{2\pi \big(u-\frac{1}{y^2}\big)^3z^3 \big(u- \big(u-\frac{1}{y^2}\big)z\big)}} \e^{-\frac{1}{2\left(u-\frac{1}{y^2}\right)z}} \D z
+ \int_{u}^\infty \frac{2}{\sqrt{2\pi \big(s-\frac{1}{y^2}\big)^3}} \e^{-\frac{1}{2\left(s-\frac{1}{y^2}\right)}} \D s.
\end{align*}
Noting $\frac{3}{4} u \le s-\frac{1}{y^2} \le s$ for all $s \ge u$, the dominated convergence theorem yields $$\int_{u y^2}^\infty \int_0^{t-1} \frac{y}{\sqrt{2\pi x^3 (t-x)^3}} \e^{-\frac{y^2}{2x}} \D x \D t \xrightarrow{y \to \infty} 0+\int_{u}^\infty \frac{2}{\sqrt{2\pi s^3}} \e^{-\frac{1}{2s}}\D s.$$
Regarding the second double integral in \eqref{4aprf}, two linear substitutions lead to
\begin{align*}
\int_{u y^2}^\infty \int_{t-1}^t \frac{y}{\sqrt{2\pi x^3 (t-x)}} \e^{-\frac{y^2}{2x}} \D x \D t
=& \int_{u}^\infty \int_{sy^2-1}^{sy^2} \frac{y^3}{\sqrt{2\pi x^3 (sy^2-x)}} \e^{-\frac{y^2}{2x}} \D x \D s\\
=& \int_{u}^\infty \int_{0}^{1} \frac{y^3}{\sqrt{2\pi (sy^2-1+z)^3 (1-z)}} \e^{-\frac{y^2}{2(sy^2-1+z)}} \D z \D s\\
=& \int_{u}^\infty \int_{0}^{1} \frac{1}{\sqrt{2\pi (s-\frac{1}{y^2}+\frac{z}{y^2})^3 (1-z)}} \e^{-\frac{1}{2\left(s-\frac{1}{y^2}+\frac{z}{y^2}\right)}} \D z \D s.
\end{align*}
As above, we can apply the dominated convergence theorem to deduce
\begin{align*}
\int_{u y^2}^\infty \int_{t-1}^t \frac{y}{\sqrt{2\pi x^3 (t-x)}} \e^{-\frac{y^2}{2x}} \D x \D t
\xrightarrow{y \to \infty}
\int_{u}^\infty\int_0^1 \frac{1}{\sqrt{2\pi s^3 (1-z)}} \e^{-\frac{1}{2s}} \D z \D s
= \int_{u}^\infty \frac{2}{\sqrt{2\pi s^3}} \e^{-\frac{1}{2s}}\D s.
\end{align*}
Combining the two limits with \eqref{4aprf}, we get $$\p\left(\frac{g^y}{y^2} \gne u\,\bigg\abs\, g^y \gne 0\right) \xrightarrow{y \to \infty} \int_{u}^\infty \frac{1}{\sqrt{2\pi s^3}} \e^{-\frac{1}{2s}}\D s$$ proving the first claim. The second claim follows immediately.
\\ \\
\textit{Part (b):} Similar to part (a), it suffices to prove the convergence of $y^2(1-\Gamma^y)$ and $y^2(1-g^y)$.
Let $u \ge 0$ and $y \lne -\sqrt{u}$. Substituting $t= \frac{y^2}{y^2+2z}$ or equivalently $z= -\frac{y^2(t-1)}{2t}$ and applying the dominated convergence theorem, we obtain
\begin{align*}
&\int_0^{1-\frac{u}{y^2}} \frac{y^2}{2\sqrt{t}} \e^{\frac{y^2}{2}} \e^{-\frac{y^2}{2t}} \D t
= \int_{\frac{u}{2}\left(1-\frac{u}{y^2}\right)^{-1}}^\infty  \frac{y^2\sqrt{y^2+2z}}{2\abs y\abs} \e^{-z} \frac{2y^2}{(y^2+2z)^2} \D z\\
=& \int_{0}^\infty \1_{\left\{\frac{u}{2}\le \left(1-\frac{u}{y^2}\right)z \right\}} \left(\frac{y^2}{y^2+2z}\right)^{\frac{3}{2}} \e^{-z} \D z
\to \int_{0}^\infty \1_{\left\{\frac{u}{2}\le z \right\}} \e^{-z} \D z = \e^{-\frac{u}{2}}, \quad y \to -\infty.
\end{align*}
We deduce
\begin{align*}
\p( y^2(1-\Gamma^y) \ge u)
=& \p\left(\Gamma^y \le 1-\frac{u}{y^2}\right)
\stackrel{\eqref{GammaForm}}{=} \frac{\int_0^{1-\frac{u}{y^2}} \frac{1}{\sqrt{t}} \e^{-\frac{y^2}{2t}} \D t}{\int_0^1 \frac{1}{\sqrt{t}} \e^{-\frac{y^2}{2t}} \D t}\\
=&  \frac{\int_0^{1-\frac{u}{y^2}} \frac{y^2}{2\sqrt{t}} \e^{\frac{y^2}{2}} \e^{-\frac{y^2}{2t}} \D t}{\int_0^1 \frac{y^2}{2\sqrt{t}} \e^{\frac{y^2}{2}} \e^{-\frac{y^2}{2t}} \D t}
\to
\e^{-\frac{u}{2}},
\quad y \to -\infty,
\end{align*}
proving the claimed convergence of $y^2(1-\Gamma^y)$. Recalling $\p(g^y \lne \infty)=1$, we similarly get
\begin{align*}
\p(y^2(1-g^y)\le u) 
\stackrel{\hphantom{\eqref{q=1}}}{=}&
1- \p\left(g^y \lne 1-\frac{u}{y^2}\right)
\stackrel{\eqref{gForm}}{=} 1-\frac{\int_0^{1-\frac{u}{y^2}} q^y(t,1) \frac{1}{\sqrt{t}} \e^{-\frac{y^2}{2t}} \D t}{2\int_0^1 \frac{1}{\sqrt{t}} \e^{-\frac{y^2}{2t}} \D t} \\
\stackrel{\eqref{q=1}}{=}& 1-\frac{\int_0^{1-\frac{u}{y^2}} \frac{1}{\sqrt{t}} \e^{-\frac{y^2}{2t}} \D t}{2\int_0^1 \frac{1}{\sqrt{t}} \e^{-\frac{y^2}{2t}} \D t}
\to 1- \frac{1}{2}\e^{-\frac{u}{2}}
= \p(g' \le u), \quad y \to -\infty.
\end{align*}
Now let $u \le 0$ and $y \lne 0$. Substituting $x= \frac{y^2}{y^2+2z}$ (as above) and applying the dominated convergence theorem with majorant $\frac{4z}{\sqrt{2\pi(2z-u)}}\e^{-z}$, we obtain
\begin{align*}
\int_{1-\frac{u}{y^2}}^\infty q^y(t,1) \frac{y^2}{2 \sqrt{t}} \e^{\frac{y^2}{2}} \e^{-\frac{y^2}{2t}}\D t
\stackrel{\eqref{qFormEqn}}{=}& \int_{1-\frac{u}{y^2}}^\infty \int_0^1 \frac{(1-x)\abs y\abs y^2}{2\sqrt{2\pi x^3 (t-x)^3}} \e^{\frac{y^2}{2}} \e^{-\frac{y^2}{2x}} \D x \D t\\
\stackrel{\hphantom{\eqref{qFormEqn}}}{=}& \int_0^1 \frac{(1-x)\abs y\abs y^2}{\sqrt{2\pi x^3 \big(1-\frac{u}{y^2}-x\big)}} \e^{\frac{y^2}{2}} \e^{-\frac{y^2}{2x}} \D x\\
\stackrel{\hphantom{\eqref{qFormEqn}}}{=}& \int_0^\infty \frac{2z\abs y\abs y^2\sqrt{ (y^2+2z)^3}}{(y^2+2z)\sqrt{2\pi y^6\big(1-\frac{u}{y^2}-\frac{y^2}{y^2+2z}\big)}} \e^{-z} \frac{2y^2}{(y^2+2z)^2} \D z\\
\stackrel{\hphantom{\eqref{qFormEqn}}}{=}& \int_0^\infty \frac{4z}{\sqrt{2\pi\big(2z-u \frac{y^2+2z}{y^2}\big)}} \e^{-z} \frac{y^2}{y^2+2z} \D z\\
\stackrel{\hphantom{\eqref{qFormEqn}}}{\to}& \int_0^\infty \frac{4z}{\sqrt{2\pi(2z-u)}} \e^{-z} \D z, \quad y \to -\infty.
\end{align*}
Recalling $\p(g^y\lne \infty)=1$, we can apply~\eqref{gForm} to deduce
\begin{align*}
\p(y^2(1-g^y)\le u)
=& \p\left(g^y \gne 1-\frac{u}{y^2}\right)
\stackrel{\eqref{gForm}}{=}
\frac{\int_{1-\frac{u}{y^2}}^\infty q^y(t,1) \frac{1}{\sqrt{t}} \e^{-\frac{y^2}{2t}}\D t}{2\int_0^1 \frac{1}{\sqrt{t}} \e^{-\frac{y^2}{2t}} \D t}\\
=& \frac{\int_{1-\frac{u}{y^2}}^\infty q^y(t,1) \frac{y^2}{2 \sqrt{t}} \e^{\frac{y^2}{2}} \e^{-\frac{y^2}{2t}}\D t}{2\int_0^1 \frac{y^2}{2\sqrt{t}} \e^{\frac{y^2}{2}} \e^{-\frac{y^2}{2t}} \D t}
\to \frac{\int_0^\infty \frac{4z}{\sqrt{2\pi(2z-u)}} \e^{-z} \D z}{2}
= \p(g' \le u), \quad y \to -\infty,
\end{align*}
proving $y^2(1-g^y) \Rightarrow g'$.
\end{proof}

\end{document}